\documentclass[12pt]{amsart}
\usepackage{amsmath, amssymb}
\usepackage{graphicx,url,subfig}
\RequirePackage[colorlinks,citecolor=blue,urlcolor
=green]{hyperref}
\usepackage[top=1in, bottom=1in, left=1in, right=1in]{geometry}
\newtheorem{theorem}{Theorem}[section]
\newtheorem{Lemma}[theorem]{Lemma}

\theoremstyle{definition}
\newtheorem{Definition}[theorem]{Definition}
\newtheorem{Theorem}[theorem]{Theorem}
\newtheorem{Remark}[theorem]{Remark}
\newtheorem{Proposition}[theorem]{Proposition}

\newtheorem{Corollary}[theorem]{Corollary}
\newtheorem{example}[theorem]{Example}

\def\RR{\mathbb{R}}
\def\EE{\mathbb{E}}
\def\CC{\mathbb{C}}

\theoremstyle{remark}

\numberwithin{equation}{section}

\begin{document}

\title{Fractional diffusion in Gaussian noisy environment}

\author{Guannan Hu}
\address{Department of Mathematics,  University of Kansas,  1460 Jayhawk Blvd,  Lawrence,  Kansas 66045-7594,  USA}
\email{g040h641@ku.edu}

\author{Yaozhong Hu}
\address{Department of Mathematics,  University of Kansas,  1460 Jayhawk Blvd,  Lawrence,  Kansas 66045-7594,  USA}
\email{yhu@ku.edu}
\thanks{}





\begin{abstract}
We study the
fractional diffusion in a Gaussian noisy environment
as described by the fractional order stochastic
partial equations of the following form:
$D_t^\alpha u(t, x)=\textit{B}u+u\cdot W^H$,
where $D_t^\alpha$ is the fractional derivative
of order $\alpha$ with respect to the time variable
$t$,   $\textit{B}$ is a second order   elliptic operator with respect
to the space variable $x\in\mathbb{R}^d$,  and $W^H$ a fractional Gaussian noise
of Hurst parameter $H=(H_1,  \cdots,  H_d)$.  We obtain conditions satisfied by
 $\alpha$ and  $H$ so that the square integrable
 solution $u$ exists uniquely .
\end{abstract}

\maketitle
\section{Introduction}

In recent years there have been a great amount of works on anomalous  diffusions
in the study of biophysics    and so on (see for example \cite{BIKMSBG}, \cite{phy1},
\cite{phy2}, \cite{phy3}   to mention  just a few).   In mathematics, some of these anomalous
diffusions (such as sub-diffusions)  can be described by the so-called
fractional order diffusion processes. As for the term  ``fractional order diffusion",  one has to distinguish two completely different types. One is the equation
of the form $\partial_t u(t,x)=-(-\Delta)^{\alpha} u(t,x)$,  where
$t\ge 0$, $x\in \RR^d$, $\alpha\in (0, 2)$ is a positive number,
 $\partial_t=\frac{\partial}{\partial t}$
and $\Delta =\sum_{i=1}^d \partial _{x_i}^2$ is the Laplacian.
This equation is not associated with the anomalous diffusion. Instead, it is
associated with the so-called stable process (or in general L\'evy  process),
 which has jumps.  Another equation is of the form
 $D^{(\alpha)}_tu(t, x) = \Delta u(t, x)$,  where
 $ D^{(\alpha)}_t$ is the  Caputo fractional derivative with respect to $t$
 (see \cite{skm} for the study  of various fractional
 derivatives).
 This equation is relevant to the anomalous diffusion which  we mentioned
 and has been studied by a number of researchers,
 see for example \cite{ref-journal1} and the references therein.
 If one considers the anomalous diffusion in a random environment, then
 it is naturally led to the study of a fractional order stochastic partial differential
 equation of the form $D^{(\alpha)}_tu(t, x) = B u(t, x)+u(t,x)\dot W(t,x)$,  where
 $B$ is a second order differential operator, including the Laplacian as a special
 example,  and $\dot W$  is a noise.  In this paper, we shall study this
 fractional order stochastic partial differential
 equation when $\dot W(t,x)=\dot W^H(x)$ is a time homogeneous
 fractional Gaussian noise of Hurst parameter $H=(H_1, \cdots, H_d)$.
 Mainly, we shall find a  relation between $\alpha$ and $H$ such that
 the solution to the above equation has a unique square integrable solution.

If $\alpha$ is formally set to $1$, then the above stochastic partial differential
equation has been studied in \cite{ref-journal2}. So, our work can be considered as
an  extension of the work \cite{ref-journal2} to the case of
fractional diffusion (in Gaussian noisy environment).   Let us also mention
that when we formally set $\alpha=1$, we recover one of the main results
in \cite{ref-journal2}.  Thus,  our condition
\eqref{e.cond.main}  given below    is also optimal.

Here is the organization of the paper, Section 2 will describe
the operator $B$ and the noise $W^H$ and state the main result
of the paper.  In our proof we need to use the properties of the two fundamental
solutions (Green's functions) $Z(t,x,\xi)$ and $Y(t,x,\xi)$ associated with the equation
$D^{(\alpha)}_tu(t, x) = B u(t, x)$, which is represented by the Fox's {\sl H}-function .
  We shall recall some most relevant
  results on  the   {\sl H}-function  and  the Green's function $Z(t,x,\xi)$
  and $Y(t,x,\xi)$ in Section 3.  A number of preparatory lemmas are
   needed to prove main results and they are presented in Section 4.
Finally, the last section  is devoted to the proof of our main theorem.

\section{Main result}

Let
\[
\textit{B}=\sum^d_{i, j=1}a_{i, j}(x)\frac{\partial^2}{\partial x_i\partial x_j}+\sum^d_{j=1}b_j(x)\frac{\partial}{\partial x_j}+c(x)
\]
be  a uniformly elliptic second-order differential operator with bounded continuous real-valued coefficients.
Let $u_0$ be a given  bounded continuous function (locally H\"{o}lder continuous if $d>1$).   Let $\{ W^H(x)\,,
x\in \RR^d\} $
be a time homogeneous (time-independent) fractional Brownian  field on some probability space $(\Omega, \mathcal{F} \,, P)$ (Like elsewhere in probability theory, we omit the dependence of
$W^H(x)=W^H(x,\omega)$ on $\omega\in \Omega$). Namely, the stochastic process
$\{ W^H(x)\,,
x\in \RR^d\} $ is a (multi-parameter) Gaussian process with mean $0$ and its covariance is given by
\begin{equation}
\EE \left( W^H(x) W^H(y)\right)=  \prod_{i=1}^d R_{H_i}(x_i, y_i)\,,
\end{equation}
where $  H_1, \cdots, H_d$ are some real numbers in the interval $(0, 1)$.
Due to some technical difficulty, we assume that $H_i>1/2$ for all $i=  1, 2, \cdots, d$; the symbol
$\EE$ denotes the expectation on $(\Omega, \mathcal{F} \,, P)$ and
\[
R_{H_i} (x_i,y_i)=\frac12 \left(|x_i|^{2H_i}+|y_i|^{2H_i}-|x_i-y_i|^{2H_i}\right) \,,\quad \forall \
x_i,y_i\in \RR
\]
is the covariance function of a fractional Brownian motion of Hurst parameter $H_i$.

Throughout this paper
we fix an arbitrary parameter $\alpha \in (0, 1)$ and a finite time horizon $T\in (0, \infty)$.
We   study the following stochastic partial differential
equation of fractional order:
\begin{equation}\label{1.4}
\begin{cases}
D^{(\alpha)}_tu(t, x)
= \textit{B}u(t, x) + u(t, x) \cdot \dot W^H( x),  \qquad t\in (0, T],  \quad  x\in\mathbb{R}^d\, ;\\
  u(0, x)= u_0(x)\,,
\end{cases}
\end{equation}
where
\[
D^{(\alpha)}_tu(t, x)=\frac{1}{\gamma(1-\alpha)}\bigg[\frac{\partial}{\partial t}\int^t_0(t-\tau)^{-\alpha}u(\tau, x) d\tau - t^{-\alpha}u(0, x)\bigg]\quad
\]
 is the  Caputo fractional derivative (see e.g. \cite{skm}) and $\dot W^H(x) =\frac{\partial ^{d }}{ \partial x_1\cdots\partial x_d} W^H(x)$ is the distributional derivative (generalized derivative) of $W^H$, called fractional Brownian noise.

Our objective is to obtain condition  on $\alpha$ and $H$ such that the above
equation has a unique solution.  But since $W^H$ is not differentiable or since $\dot W^H(x)$ does
not exist as an ordinary function,  we have to describe under what sense   a random
field $\left\{ u(t,x)\,, t\ge 0\,, x\in \RR^d\right\}$ is a solution to the above equation
(\ref{1.4}).

To motivate our definition of the solution, let us consider the following (deterministic) partial
differential equation of fractional order with the term $u(t, x) \cdot \dot W^H(x)$
in (\ref{1.4}) replaced by $f(t,x)$:
\begin{equation}
\label{1.1}
\begin{cases}
D^{(\alpha)}_t\tilde u(t, x)
= \textit{B}\tilde u(t, x) + f(t,x),  \qquad t\in (0, T],  \quad  x\in\mathbb{R}^d\, ;\\
 \tilde  u(0, x)= u_0(x)\,,
\end{cases}
\end{equation}
where the function $f$ is bounded and jointly continuous in $(t, x)$ and locally H\"{o}lder continuous in x.
In   \cite{ref-journal1},  it is proved that there are two Green's functions
$$\left\{ Z(t , x,  \xi) \,,\  Y(t , x,  \xi)\,,  0 <t\le T\,,
 x, \xi\in \RR^d\right\}$$  such that the solution to the  Cauchy problem   \eqref{1.1} is
 given by
\begin{equation}
\label{1.2} \tilde
u(t, x)=\int_{\mathbb{R}^d} Z(t, x,  \xi)u_0(\xi)d\xi+\int_0^t  ds \int_{\mathbb{R}^d} Y(t-s, x, y)f(s, y) dy.
\end{equation}
In general, there is no  explicit form for  the two Green's functions
$\left\{ Z(t , x,  \xi) \,,\  Y(t, x,  \xi) \right\}$. However, their constructions and
properties are known (see  \cite{ref-journal1}, \cite{ref-journal3}, \cite{ref-journal4}, and the references
therein). We shall recall some  needed results   in the next section.

From the classical solution expression   (\ref{1.2}),   we expect that the solution
$u(t,x)$ to (\ref{1.4})  satisfies formally
\[
u(t, x)=\int_{\mathbb{R}^d} Z(t, x,  \xi)u_0(\xi)d\xi+\int_0^t  ds \int_{\mathbb{R}^d} Y(t-s, x, y)u(s, y) \dot W^H(  y) dy\,.
\]
   The above formal integral
$ \int_0^t  ds \int_{\mathbb{R}^d} Y(t-s, x, y)u(s, y) \dot W^H(  y) dy$ can be defined by It\^o-Skorohod stochastic
integral $\int_{\RR^d}\left[\int_0^t    Y(t-s, x, y) u(s, y) ds\right] W^H(  dy)  $  as given  in
\cite{ref-journal2}.

Now,   we can give the following definition.

\begin{Definition}\label{d.mild}  A random field $\left\{u(t,x)\,,  0\le t\le T\,,
 x \in \RR^d\right\}$ is called a  mild  solution to the equation
 (\ref{1.4})   if
\begin{enumerate}
  \item\ $u(t, x)$ is jointly measurable in $t\in [0, T]$ and $x\in\RR^d$;
  \item\  $\forall (t, x)\in [0, T] \times \mathbb{R}^d$,  $\int_{0}^t\int_{\mathbb{R}^d}Y(t-s, x, y)u(s, y)ds W^H(dy) $ is well defined in $\mathcal{L}^2$;
  \item\  The following holds in $\mathcal{L}^2$
  \begin{equation}\label{e.2.5}
 u(t, x)=\int_{\mathbb{R}^d}Z(t, x,\xi)u_0(\xi)d\xi+\int_{0}^t\int_{\mathbb{R}^d}Y(t-s, x, y)u(s, y)W^H(dy)d s.
  \end{equation}
\end{enumerate}
\end{Definition}

Let us return to the discussion of the two Green's functions
$\left\{ Z(t , x,  \xi) \,,\  Y(t , x,  \xi) \right\}$.
If $\alpha=1$, namely, if
 $D^{(\alpha)}_t$ in (\ref{1.1})   is replaced by $\partial _t$ and $
 \displaystyle B= \Delta:=  \sum_{i=1}^d \partial _{x_i}^2$,
 then
\begin{equation}
 Z(t , x,  \xi) = Y(t , x,  \xi)=
    \left(4\pi t\right)^{-d/2} \exp\left\{ -\frac{|x-\xi|^2}{
4t}\right\}\,.\label{e.zandy.1}
\end{equation}
In this case the stochastic partial differential equation
of the form
\begin{equation}\label{1.3}
 \frac{\partial u(t, x)}{\partial t}= \Delta u(t, x) + u \cdot W^H(x), \qquad x\in\mathbb{R}^d,
 \end{equation}
 was studied in \cite{ref-journal2}.  The mild solution to the above equation
 (\ref{1.3}) is proved to exist uniquely  under conditions
\begin{equation}
H_i>1/2\,,\quad i=1, \cdots, d\qquad {\rm and}\qquad \sum_{i=1}^d H_i>d-1\,.\label{e.cond.classical}
\end{equation}

 The main result of this paper is to extend the above result in \cite{ref-journal2}
 to our equation (\ref{1.4}).
\begin{Theorem}\label{main}   Let the coefficients $ a_{ij}(x)$, $b_i(x) \,, i,j=1, \cdots, d\,, $ be bounded and
continuous   and let them be H\"older continuous with exponent $\gamma$.
  Let $ a_{ij}(x)$ be uniformly elliptic.  Namely, there is a
constant $a_0\in (0, \infty)$ such that
\[
\sum_{i,j=1}^d a_{ij}(x) \xi_i\xi_j\ge a_0 |\xi|^2 \quad \forall \ \ \xi=(\xi_1, \cdots, \xi_d)\in \RR^d\,.
\]
Let $u_0$ be  a bounded continuous  (and
locally H\"{o}lder continuous if $d>1$).\   Assume
\begin{equation}
H_i>\begin{cases}
\frac12  &\qquad \hbox{if}\  \ d=1, 2, 3, 4\\
1-\frac{2}{d}-\frac{\gamma}{2d}&\qquad \hbox{if}\ \ d\ge 5 \\
\end{cases}
\label{e.cond.main-hi}
 \end{equation}
 and
 \begin{equation}
 \sum_{i=1}^{d}H_i>d-2+\frac1\alpha \,.
\label{e.cond.main}
 \end{equation}
Then, the mild solution to (\ref{1.4}) exists uniquely in $L^2(\Omega, \mathcal{F}, P)$.
\end{Theorem}
\begin{Remark} (i)\ When $\alpha$ is formally set to $1$, the condition
\eqref{e.cond.main} is the same as the condition \eqref{e.cond.classical}
given in \cite{ref-journal2}. So, in some sense
our condition \eqref{e.cond.main}  is optimal.

(ii) Since $H_i<1$ for all $i=1,2, \cdots, d$ the condition is possible only when $\alpha >1/2$.
\end{Remark}


\section{Green's functions
$Z$ and $Y$}\label{s.fund}


\subsection{Fox's {\sl H}-function }

We  shall use  {\sl H}-function  to  express the Green's functions $Z$ and $Y$ in  Definition \ref{d.mild}.   In this subsection we recall some results about the {\sl H}-function s and the two  Green's functions.  We shall follow the presentation in \cite{ref-book1}
(see also \cite{ref-journal1} and references therein).

\begin{Definition}
Let  $m, n, p, q$ be integers such that  $0\leq m\leq q,  0\leq n\leq p$.  Let
  $a_i,  b_i\in \CC$ be complex numbers and
  let $\alpha_j, \beta_j$ be positive numbers,  $i=1, 2, \cdots ,  p; j=1, 2, \cdots  ,  q$.
  Let the   set of poles of the gamma functions $\Gamma(b_j+\beta_js)$ doesn't intersect with that of the gamma functions $\Gamma(1-a_i-\alpha_is)$,
namely,
\[
\bigg\{b_{jl}=\frac{-b_j-l}{\beta_j},  l =0, 1, \cdots\bigg\}\bigcap \bigg\{a_{ik}=\frac{1-a_i+k}{\alpha_i},  k=0, 1, \cdots\bigg\}=\emptyset\
\]
for all $i=1, 2, \cdots ,  p$ and $ j=1, 2, \cdots ,  q$.
The {\sl H}-function
\[
H^{mn}_{pq}(z)\equiv H^{mn}_{pq}\bigg[z \bigg|\begin{array}{ccc}
                                                                        (a_1, \alpha_1) & \cdots & (a_p, \alpha_p)\\
                                                                        (b_1, \beta_1) & \cdots & (b_q, \beta_q)
                                                                      \end{array} \bigg]
\]
is defined by   the following integral
\begin{equation}\label{2.1}
H^{mn}_{pq}(z)=\frac{1}{2\pi i}\int_L\frac{\prod_{j=1}^m \Gamma(b_j+\beta_js)\prod_{i=1}^n\Gamma(1-a_i-\alpha_is)}{\prod_{i=n+1}^p\Gamma(a_j+\alpha_is)\prod_{j=m+1}^q \Gamma(1- b_j-\beta_js)}z^{-s} ds\,, \ \ z\in \CC\,,
\end{equation}
where an empty product in \eqref{2.1}  means  $1$  and
$L$ in \eqref{2.1} is the infinite contour which separates all the points  $b_{jl}$ to the left and all the points
 $a_{ik}$ to the right of $L$.  Moreover, $L$  has one of the following forms:
\begin{enumerate}
          \item $L=L_{-\infty}$ is a left loop situated in a horizontal strip starting at point $-\infty+i\phi_1$ and terminating at point $-\infty+i\phi_2$ for some   $-\infty<\phi_1< \phi_2<\infty$
          \item $L=L_{+\infty}$ is a right loop situated in a horizontal strip starting at point $+\infty+i\phi_1$ and terminating at point $\infty+i\phi_2$ for some  $-\infty<\phi_1< \phi_2<\infty$
          \item $L=L_{i\gamma\infty}$ is a contour starting at point $\gamma-i\infty$ and terminating at point $\gamma+i\infty$ for some  $\gamma\in(-\infty,  \infty)$
        \end{enumerate}

\end{Definition}
To illustrate $L$ we give the following graphs.
\begin{table}[htbp]
\begin{tabular}{ccc}
\includegraphics[width=2.3in]{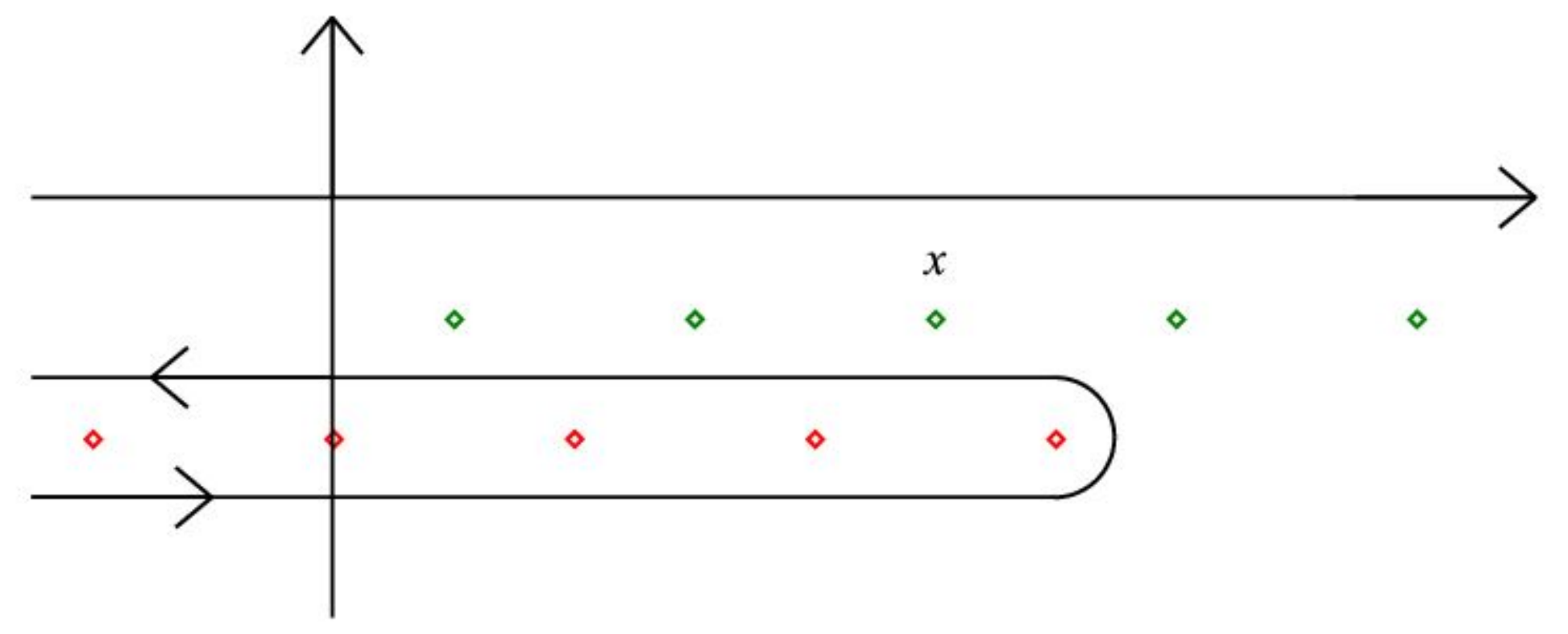}& \includegraphics[width=2.3in]{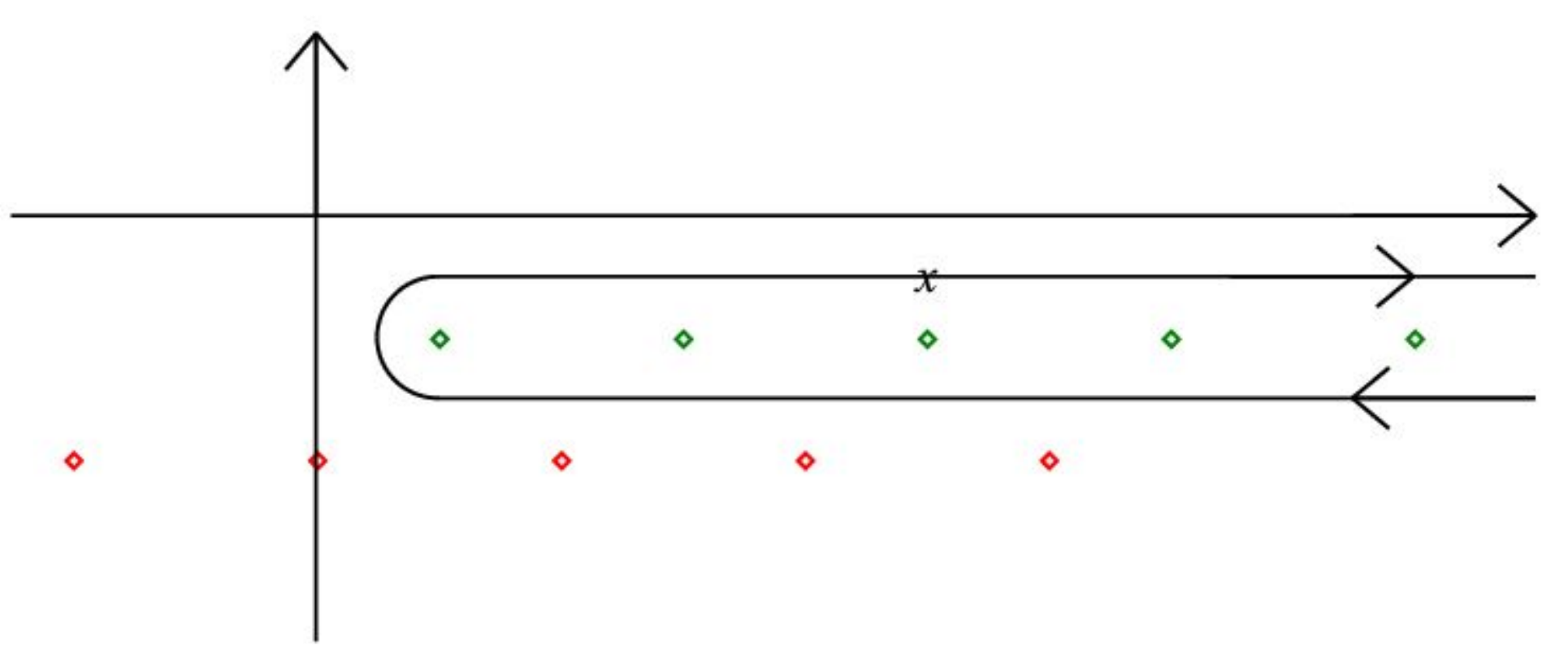}& \includegraphics[width=2.3in]{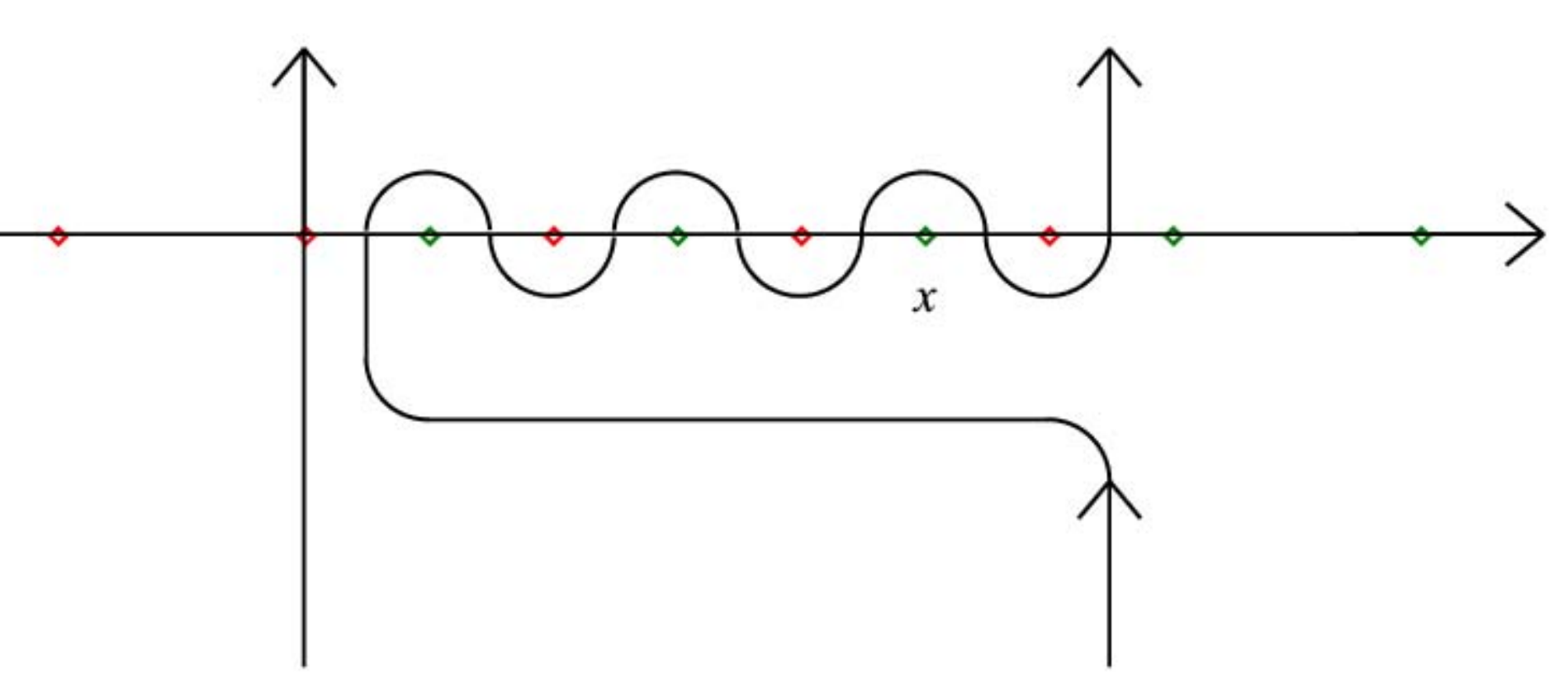}\\
Case 1&Case 2&Case 3\\
\end{tabular}
\end{table}

The integral \eqref{2.1} exists   when $\sum_{j=1}^q\beta_j-\sum_{i=1}^p\alpha_i\geq0$
(see \cite{ref-book1},  Theorem 1.1).

\begin{example} To compare with the classical case $\alpha=1$,   we consider the case
$m=2$, $n=0$, $p=1$,  $q=2$, $a_1=\alpha_1 =b_2= \beta_1=\beta_2=1$ and  $b_1=\frac d2$.
Let $L=L_{-\infty}$. Then, we have
\begin{eqnarray}
H^{20}_{12}\bigg[z\bigg|\begin{array}{cc}
                                                                                                (1, 1)&  \\
                                                                                                (\frac d2, 1), &(1, 1)
                                                                                              \end{array}
\bigg]
&=&  \frac {1}{2\pi i}\int_L\frac{ \Gamma(\frac d2+s)\Gamma(1+s)}{\Gamma(1+s)}z^{-s} ds\nonumber \\
&=&  \frac {1}{2\pi i}\int_L  \Gamma(\frac d2+s)  z^{-s} ds\nonumber \\
&=&\sum_{v=0}^{\infty}
\lim\limits_{s\rightarrow-(\frac d2+v)}(s+\frac d2+v)\Gamma(\frac d2+s)z^{-s}\nonumber \\
&=&\sum_{v=0}^{\infty}\lim\limits_{s\rightarrow-(\frac d2+v)}\frac{\Gamma(v+\frac d2+s+1)}{(s+\frac d2+v-1)\cdots (s+\frac d2)}z^{-s}\nonumber \\
&=&\sum_{v=0}^{\infty}z^{d/2}(-1)^v \frac{1}{v!}z^{v}\nonumber \\
&=&z^{d/2} \exp (-z)\,.\label{e.h.1}
\end{eqnarray}
\end{example}

\subsection{Green's functions $Z$ and $Y$ when $B$ has  constant coefficients}
%

In this subsection let us consider $Z$ and $Y$ when the operator $B$   in \eqref{1.4} has the following form
$$\textit{B}=\sum^d_{i, j=1}a_{ij}\frac{\partial^2}{\partial x_i\partial x_j}, $$
 where the matrix $A=(a_{ij})$ is positive definite. In this case, $Z$ and $Y$ (we call them $Z_0$ and $Y_0$
 to distinguish with the general coefficient case) are given as follows.
  $$Z_0(t , x )=\frac{\pi^{-d/2}}{(\det A)^{1/2}}\bigg[\sum^d_{i, j=1}A^{(ij)} x_i x_j
  \bigg]^{-d/2}$$
$$\times H^{20}_{12}\bigg[\frac14 t^{-\alpha}\sum^d_{i, j=1}A^{(ij)} x_i x_j
\bigg|\begin{array}{cc}
                                                                                                (1, \alpha)&  \\
                                                                                                (\frac d2, 1), &(1, 1)
                                                                                              \end{array}
\bigg], $$ where $(A^{(ij)})=A^{-1}$
and%
$$Y_0(t, x )=\frac{\pi^{-d/2}}{(\det A)^{1/2}}\bigg[\sum^d_{i, j=1}A^{(ij)} x_i x_j \bigg]^{-d/2} t
 ^{\alpha-1}$$
$$\times H^{20}_{12}\bigg[\frac14 t^{-\alpha}\sum^d_{i, j=1}A^{(ij)} x_i x_j \bigg|\begin{array}{cc}
                                                                                                (\alpha, \alpha)&  \\
                                                                                                (\frac d2, 1), &(1, 1)
                                                                                              \end{array}
\bigg].$$
It is easy to see that for the constant coefficient case,
the both of the Green's functions are homogeneous in time and space. Namely,
\[
Z_0(t, x,    \xi)=Z_0(t ,x-
  \xi) \,,\quad Y_0(t, x,    \xi)=Y_0(t ,x-
  \xi) \,.
\]
In particular,   when $\alpha=1$,  it is easy to see from the above expression
and the explicit form \eqref{e.h.1} of $H^{20}_{12}(z)$ that
\[
Z_0(t, x,    \xi)=Y_0(t, x,    \xi)=(4\pi)^{-d/2} \det(A)^{-1/2} \exp\left\{ -\frac{\sum^d_{i, j=1}A^{(ij)}(x_i-\xi_i)(x_j-\xi_j)}{4t }\right\}\,.
\]
which reduces to \eqref{e.zandy.1} when $A=I$ is the identity matrix.
%
%
 %
%

%

With the above expression for $Z_0$ and $Y_0$ and the properties of the {\sl H}-function , one can obtain the following estimates.

%
%
%

\begin{Proposition} 
Denote
\begin{equation}
p(t, x)=\exp\big(-\sigma t^{-\frac{\alpha}{2-\alpha}}|x|^{\frac{2}{2-\alpha}}\big)\,, \quad t>0\,, \
x\in \RR^d\,,  \label{e.p-x}
\end{equation}
where $\sigma\in (0, \infty)$ is a constant whose exact value is irrelevant
in  the paper. Then,     we have the following estimates:
\begin{equation}
|Z_0(t, x)|\leq
\begin{cases}
Ct^{-\frac{\alpha}{2} }  p(t, x)& \hbox{when}\quad d=1\\
Ct^{-\alpha }  [|\log\frac{|x|^2}{t^{\alpha}}|+1] p(t, x)& \hbox{when}\quad d=2\\
 Ct^{-\alpha}|x|^{2-d}  p(t, x)& \hbox{when}\quad d\ge 3\,,
  \end{cases}\label{e.ine.z0}
 \end{equation}
 where for instance,
 $|Z_0(t, x)|\leq Ct^{-\frac{\alpha}{2} }  p(t, x)$ means that there are  positive constant $C$ and positive constant
$\sigma$ such that the above inequality holds. In what follows  the positive
constants $C$ and $\sigma$ are generic,  which may be different in
different appearances.
\end{Proposition}

\begin{proof}
Denote $R=x^2/t^{\alpha}$.
From \cite{ref-journal1},  Proposition 1, it follows that
when $R\le 1$,  we  have
\[
|Z_0(t, x)|  \leq
\begin{cases}
Ct^{-\frac{\alpha}{2} }     & \hbox{when}\quad d=1\\
Ct^{-\alpha }  [|\log\frac{|x|^2}{t^{\alpha}} |+1]   & \hbox{when} \quad d=2\\
 Ct^{-\alpha}  |x|^{2-d}    & \hbox{when}\quad d\ge 3\,,    \\
\end{cases}\
\]
Since when $R\le 1$,  $p(t, x)$ is bounded from below. This proves
the inequality \eqref{e.ine.z0}
when $R\le 1$.

When $R>1$,  then by  \cite{ref-journal1},  Proposition 1 we have
$|Z_0(t, x)|  \leq  C t^{-\frac{\alpha d} {2}}  p(t,x)$.  It is clear that this implies
the inequality \eqref{e.ine.z0}   when $d=1$ and $d=2$.  Now, we assume that $d\ge 3$.
  We have
\begin{eqnarray*}
|Z_0(t, x)|  &\leq&  C t^{-\frac{\alpha d} {2}}  p(t,x)
\le Ct ^{-\alpha } x^{2-d} \left(\frac{x^2}{t^{\alpha}}\right)^{\frac{d}{2}-1} p(t,x)\\
&\le& Ct ^{-\alpha } x^{2-d}   p(t,x)\,,
\end{eqnarray*}
where we used the fact that $\left(\frac{x^2}{t^{\alpha}}\right)^{\frac{d}{2}-1} p(t,x)\le p(t,x)$
for a different $\sigma$ in the later $p(t,x)$.
\end{proof}

Similarly,   we can use  \cite{ref-journal1},   Proposition 2 (for $d=1$ case)  and
  \cite{ref-journal1}, Section 4.2 (for   $d\ge 2$ case) to
  obtain  the following estimates for $Y_0(t,x)$.
\begin{Proposition}\label{p.6}  
We follow the same notation $p(t,x)$ as defined by \eqref{e.p-x}.  We have
\begin{enumerate}
\item\quad
When $d=1$,  we have the following estimates:
\begin{eqnarray}
|Y_0(t, x)|\leq \left\{
  \begin{array}{ll}
    $$Ct^{\frac{\alpha}{2}-1}p(t, x)$$   & \hbox{when}\quad \hbox{$t^{-\alpha}|x|^2\geq1$ } \\
    $$Ct^{\frac{\alpha}{2}-1}$$  & \hbox{when}\quad \hbox{$t^{-\alpha}|x|^2\leq1$.}
  \end{array}
\right.
\end{eqnarray}
\item\quad  When $d\ge 2$, we have the following estimates:
 \begin{equation}
|Y_0(t, x)|\leq
\begin{cases}
Ct^{-1}  p(t, x)& \hbox{when}\quad d=2\\
Ct^{-\frac{\alpha}{2}-1}  p(t, x)& \hbox{when}\quad d=3\\
 Ct^{-\alpha-1}[|\log\frac{|x|^2}{t^{\alpha}}|+1]  p(t, x)& \hbox{when}\quad d=4\\
 Ct^{-\alpha-1}|x|^{4-d}  p(t, x)& \hbox{when}\quad d\ge 5\,,   \\
 \end{cases}
 \end{equation}
where for instance, $|Y_0(t, x)|\leq
Ct^{-1}  p(t, x)$ means that there are  positive constant $C$ and positive constant
$\sigma$ such that the above inequality holds. In what follows  the positive
constants $C$ and $\sigma$ are generic,  which may be different in
different appearances.
 \end{enumerate}
\end{Proposition}

\subsection{Green's functions $Z$ and $Y$ in general coefficient case}

If  the coefficients of $B$ are not constant, then the Green's functions $Z$ and $Y$
are more complicated and may be obtained by a method similar to the Levi parametrix
for the parabolic equations.

Denote
\begin{eqnarray*}
M(t, x, \xi)
&=&\sum^d_{i, j=1} [a_{ij}(x)-a_{ij}(\xi)]\frac{\partial^2}{\partial x_i\partial x_j}Z_0(t, x-\xi) \\
&&\qquad +\sum^d_{i=1}b_{i}(x)\frac{\partial}{\partial x_i}Z_0(t, x-\xi)+c(x)Z_0(t, x-\xi)\\
K(t, x, \xi)
&=&\sum^d_{i, j=1} [a_{ij}(x)-a_{ij}(\xi)]\frac{\partial^2}{\partial x_i\partial x_j}Y_0(t, x-\xi) \\
&&\qquad
 +\sum^d_{i=1}b_{i}(x)\frac{\partial}{\partial x_i}Y_0(t, x-\xi)+c(x)Y_0(t, x-\xi)\,.
 \end{eqnarray*}
Let
$Q(s, y, \xi)$ and $\Phi(s, y, \xi)$ be  defined by
$$Q(t, x, \xi)=M(t, x, \xi)+ \int_0^t  ds \int_{\mathbb{R}^d} K(t-s, x,y)Q(s, y, \xi) dy, $$
$$\Phi(t, x, \xi)=K(t, x, \xi)+ \int_0^t  ds \int_{\mathbb{R}^d} K(t-s, x,y)\Phi(s, y, \xi) dy$$

\begin{Proposition}\label{Proposition5}  Let the coefficients $a_{ij}(x)$ and
$b_i(x) $ satisfy the conditions in Theorem \ref{main}. Recall that $\gamma$ is the
H\"older exponent of the coefficients with respect to the spatial variable $x$.
Then, the  Green's  functions $\{Z(t, x, \xi), Y(t, x, \xi)\}$ have  the following form:
$$Z(t, x, \xi)=Z_0(t, x-\xi)+V_Z(t, x, \xi); $$
\begin{equation}\label{2.2}
Y(t, x, \xi)=Y_0(t, x-\xi)+V_Y(t, x, \xi),
\end{equation}
where $$V_Z(t, x, \xi)= \int_0^t  ds \int_{\mathbb{R}^d} Y_0(t-s, x,y)Q(s, y, \xi) dy; $$
$$V_Y(t, x, \xi)= \int_0^t  ds \int_{\mathbb{R}^d} Y_0(t-s, x,y)\Phi(s, y, \xi) dy. $$

Moreover,  the function $V_Z(t, x, \xi), V_Y(t, x, \xi)$ satisfy the following estimates.
\begin{equation}
|V_Z(t, x, \xi)|\leq
\begin{cases}
Ct^{(\gamma-1)\frac\alpha2}  p(t, x-\xi)\,, & \qquad \hbox{when } \ d=1\,;\\
Ct^{\frac{\gamma\alpha}{2}-\alpha}  p(t, x-\xi)\,,  &\qquad \hbox{when } \ d=2\,;\\
 Ct^{\frac{\gamma_0\alpha}{2}-\alpha}|x-\xi|^{2-d+\gamma-\gamma_0}  p(t, x-\xi)\,,
 &\qquad \hbox{when } \ d=3 \ {\rm or}\ d\ge 5\,; \\
Ct^{(\gamma-\gamma_0)\frac\alpha2-\alpha}|x-\xi|^{-2+ \gamma-2\gamma_0}  p(t, x-\xi)\,,
 &\qquad \hbox{when } \ d=4  \\
\end{cases}
\end{equation}
and
\begin{equation}
|V_Y(t, x, \xi)|\leq
\begin{cases} Ct^{\alpha -1+(\gamma-1)\frac\alpha2}  p(t, x-\xi)\,, & \qquad \hbox{when } \ d=1\,;\\
Ct^{\frac{\gamma\alpha}{2}-1}  p(t, x-\xi)\,, & \qquad \hbox{when } \ d=2\,;\\
Ct^{(\gamma_0+\gamma)\frac\alpha4-1}|x-\xi|^{2-d+(\gamma-\gamma_0)/2}  p(t, x-\xi)\,, & \qquad \hbox{when } \  d=3 \ {\rm or}\ d\ge 5\,; \\
Ct^{(\gamma-\gamma_0)\frac\alpha4-1}|x-\xi|^{-2+\gamma-2\gamma_0}  p(t, x-\xi)\,,
 &\qquad \hbox{when } \ d=4  \\
\end{cases}
\end{equation}
Here $\gamma_0$ is any number such that
$0 < \gamma_0<\gamma$  and in the case    $d\geq 3$,
the constant C   depends  on $\gamma_0$.

\end{Proposition}
%

\section{Auxiliary lemmas}


To prove our main theorem,  we   need to dominate certain
multiple integral involving $Y(t,x, \xi)$ and $Z(t,x,\xi)$.
Since both $Y(t,x, \xi)$ and $Z(t,x,\xi)$ are complicated,
we shall first bounded them by $p(t,x-\xi)$
from  the estimations of $|Y_0(t, x)|$ and $|V_Y(t, x,\xi)|$.
More precisely, we have the following bounds for
$Y(t,x, \xi)$.

\begin{Lemma}\label{lem5} Let $x \in  \mathbb{R}^d, t\in(0, T]$.   Then
\begin{equation}
|Y(t, x,\xi)|\leq \left\{
  \begin{array}{ll}
   $$ Ct^{-1+\frac{\alpha}{2}}   p(t, x-\xi)$$,  & \hbox{ $d=1$;} \\
     $$Ct^{-1}  p(t, x-\xi) $$,  & \hbox{$d=2$;} \\
       $$Ct^{-(\gamma-2\gamma_0)\frac\alpha2-1} |x-\xi|^{-2+\gamma-2\gamma_0}  p(t, x-\xi)$$,  & \hbox{$d=4$;} \\
     $$Ct^{-(\gamma-\gamma_0)\frac\alpha4-1}|x-\xi|^{2-d+(\gamma-\gamma_0)/2}  p(t, x-\xi)$$,  & \hbox{$d=3$
     \ {\rm or}\  $d\geq 5$.}
  \end{array}
\right. \label{e.Y-bound}
\end{equation}
\end{Lemma}
\begin{proof}  We shall prove the lemma case by case.
First, when $d=1$,  by   Proposition \ref{p.6}, we have
\begin{eqnarray*}
|Y_0(t, x-\xi)|
&\leq& \left\{
  \begin{array}{ll}
    $$Ct^{\frac{\alpha}{2}-1}  p(t, x-\xi)$$,  & \hbox{$t^{-\alpha}|x-\xi|^2\geq1$;} \\
    $$Ct^{\frac{\alpha}{2}-1}$$,  & \hbox{$t^{-\alpha}|x-\xi|^2\leq1$.}
  \end{array}
\right.
\end{eqnarray*}
If $t^{-\alpha}|x-\xi|^2\leq1$,  then
$$|Y_0(t, x-\xi)|\leq Ct^{-1+\frac{\alpha}{2}}\cdot \frac {p(x, t)}{e^{-\sigma}}\leq Ct^{\frac{\alpha}{2}-1}  p(t, x-\xi).$$
Therefore
\begin{eqnarray*}
|Y(t, x,\xi)|
&\leq& |Y_0(t, x-\xi)|+|V_Y(t, x,\xi)|\\
&\leq& Ct^{\alpha -1+(\gamma-1)\frac\alpha2}  p(t, x-\xi)+Ct^{-1+\frac{\alpha}{2}}  p(t, x-\xi)\\
&\leq& Ct^{-1+\frac{\alpha}{2}}   p(t, x-\xi)\,.
\end{eqnarray*}
Now we consider the case   $d=2$. From the following inequalities:
\begin{eqnarray*}
|V_Y(t, x,\xi)|
&\leq& Ct^{\gamma\frac{\alpha}{2}-1}  p(t, x-\xi)\,;\\
|Y_0(t, x-\xi)|
&\leq& Ct^{-1}  p(t, x-\xi)\
\end{eqnarray*}
we have easily
\[
|Y(t, x,\xi)|\leq |Y_0(t, x-\xi)|+|V_Y(t, x,\xi)|\leq Ct^{-1}  p(t, x-\xi) \,.
\]
We re going to  prove the lemma when   $d=3$.  From Proposition  \ref{p.6}
we have
\begin{eqnarray*}
|Y_0(t, x-\xi)|
&\leq& Ct^{-\frac\alpha2-1}  p(t, x-\xi)\\
&=&Ct^{-(\gamma-\gamma_0)\frac\alpha4-1}|x-\xi|^{-1+(\gamma-\gamma_0)/2}
\bigg|\frac{x-\xi}{t^\frac\alpha2}\bigg|^{1-(\gamma-\gamma_0)/2}  p(t, x-\xi)\\
&\leq& Ct^{-(\gamma-\gamma_0)\frac\alpha4-1}|x-\xi|^{-1+(\gamma-\gamma_0)/2}  p(t, x-\xi)\,.
\end{eqnarray*}
Combining this inequality with Proposition \ref{Proposition5}  we obtain
\[
|Y(t, x,\xi)|\leq Ct^{-(\gamma-\gamma_0)\frac\alpha4-1}|x-\xi|^{-1+(\gamma-\gamma_0)/2}  p(t, x-\xi)\,.
\]
We turn to consider the case  $d=4$. 
Proposition \ref{p.6} yields that for any $\theta>0$ the following holds true:
\begin{eqnarray*}
|Y_0(t, x-\xi)|
&\leq&  Ct^{-\alpha-1} \bigg[\bigg(\frac{|x-\xi|^2}{t^{\alpha}}\bigg)^\theta +\bigg(\frac{t^{\alpha}}{|x-\xi|^2}\bigg)^\theta\bigg]  p(t, x-\xi)\,;\\
&=& Ct^{-\alpha-1} \bigg(\frac{t^{\alpha}}{|x-\xi|^2}\bigg)^\theta\bigg[\bigg(\frac{|x-\xi|^2}{t^{\alpha}}\bigg)^{2\theta} +1\bigg]  p(t, x-\xi)\,.
\end{eqnarray*}

If $\frac{|x-\xi|^2}{t^{\alpha}}  >1$,   then
\begin{eqnarray*}
\bigg[\bigg(\frac{|x-\xi|^2}{t^{\alpha}}\bigg)^{2\theta} +1\bigg]  p(t, x-\xi)
&\leq&  2\bigg(\frac{|x-\xi|^2}{t^{\alpha}}\bigg)^{2\theta}  p(t, x-\xi) \leq C   p(t, x-\xi)\,.
\end{eqnarray*}
As a consequence, we have
\[
|Y_0(t, x-\xi)|
\leq Ct^{-\alpha-1} \bigg(\frac{t^{\alpha}}{|x-\xi|^2}\bigg)^\theta  p(t, x-\xi)\,.
\]
If $\frac{|x-\xi|^2}{t^{\alpha}}  \leq 1, $  then  the above inequality is obviously true.
Now,  we can choose $\theta>0$,  such that $-2\theta \geq (-2+\gamma-2\gamma_0)$.
Thus, we have
\begin{eqnarray*}
|Y_0(t, x-\xi)|
&=&  Ct^{-\alpha-1+\alpha\theta+(-2\theta-(-2+\gamma-2\gamma_0))\frac\alpha2}  |x-\xi|^{-2+\gamma-2\gamma_0} \\
&&\qquad \cdot \bigg(\frac{|x-\xi|}{t^{\frac\alpha2}}\bigg)^{-2\theta-(-2+\gamma-2\gamma_0)}  p(t, x-\xi)\\
&\leq&  Ct^{-(\gamma-2\gamma_0)\frac\alpha2-1}\cdot |x-\xi|^{-2+\gamma-2\gamma_0}  p(t, x-\xi)\,.
\end{eqnarray*}
Combining the above inequality with Proposition \ref{Proposition5} we have
\begin{eqnarray*}
|Y(t, x,\xi)|
&\leq&  Ct^{-(\gamma-2\gamma_0)\frac\alpha2-1} |x-\xi|^{-2+\gamma-2\gamma_0}  p(t, x-\xi)\\
&&\qquad +Ct^{(\gamma_0+\gamma)\frac\alpha4-1}|x-\xi|^{-2+\gamma-2\gamma_0}  p(t, x-\xi)\\
&\leq& Ct^{-(\gamma-2\gamma_0)\frac\alpha2-1} |x-\xi|^{-2+\gamma-2\gamma_0}  p(t, x-\xi)\
\end{eqnarray*}
since $-(\gamma-2\gamma_0)\frac\alpha2-1 \leq (\gamma_0+\gamma)\frac\alpha4-1.$

Finally we consider the case  $d\geq5$.  From the estimates
$|Y_0(t, x-\xi)|
\le   Ct^{-\alpha-1} |x-\xi|^{4-d}  p(t, x-\xi)$
we obtain
\begin{eqnarray*}
|Y_0(t, x-\xi)|
&\leq &   C t^{-(\gamma_0+\gamma)\frac\alpha4-1}|x-\xi|^{2-d+(\gamma-\gamma_0)/2}\bigg|\frac{x-\xi}{t^\frac\alpha2}\bigg|^{2-(\gamma-\gamma_0)/2}  p(t, x-\xi)\\
&\leq& t^{-(\gamma-\gamma_0)\frac\alpha4-1}|x-\xi|^{2-d+(\gamma-\gamma_0)/2}  p(t, x-\xi)\,.
\end{eqnarray*}
Therefore, we have
$$|Y(t, x,\xi)|\leq Ct^{-(\gamma-\gamma_0)\frac\alpha4-1}|x-\xi|^{2-d+(\gamma-\gamma_0)/2}  p(t, x-\xi)\,. $$
The proposition is then proved.
\end{proof}

The bound \eqref{e.Y-bound} will greatly help to
simplify our estimation of the multiple integrals that we are going to
encounter. However,  when the dimension
$d$  is greater than or equal to $2$, the multiple integrals are still complicated
to estimate and our main technique is to reduce the computation to one dimensional.
  This means we shall further bound the right hand side of
the inequality \eqref{e.Y-bound} by   product
of functions of one variable.   Before doing so,
we denote the exponents of $t$ and $|x-\xi|$ in \eqref{e.Y-bound}  by $\zeta_d$ and $\kappa_d$.
Namely, we denote
\begin{equation}
 \zeta_{d} =
\begin{cases} -1+\frac{\alpha}{2}  ,  &  d=1 ;  \\
      -1   ,  &  d=2 ;  \\
       -(\gamma-2\gamma_0)\frac\alpha2-1 ,  & d=4;  \\
     -(\gamma-\gamma_0)\frac\alpha4-1  ,  &  d=3\ {\rm or}\ d\geq 5\,.
  \end{cases}\label{e.zeta}
  \end{equation}
  and
\begin{equation}
 \kappa_{d} =
\begin{cases}   0 ,  &  d=1, 2   ;  \\
     -2+\gamma-2\gamma_0  ,  & d=4;  \\
     2-d+(\gamma-\gamma_0)/2   ,  &   d=3\ {\rm or}\ d\geq 5\,.
  \end{cases}\label{e.kappa}
\end{equation}
From now on we shall exclusively use $p(t,x)
=
\exp\big(-\sigma t^{-\frac{\alpha}{2-\alpha}}|x|^{\frac{2}{2-\alpha}}\big)$ to denote a function of one variable.
However, the constant $\sigma$ may be different in different appearances of $p(t,x)$ (for notational simplicity,
we omit the explicit dependence on $\sigma$ of $p(t,x)$).


With these notation  Lemma \ref{lem5}   yields
%
%


\begin{Lemma}\label{Corollary7}
The following bound holds true for the Green's function $Y$:
\begin{equation}\label{e.Y-product-b}
|Y(t, x,\xi)|\leq C\prod_{i=1}^{d} t^{\zeta_d/d} |x_i-\xi_i|^{\kappa_d/d}   p (t, x_i-\xi_i)\,.
\end{equation}
\end{Lemma}

\begin{proof} It is easy to see that
\[
|x|=\left(\sum_{i=1}^d x_i^2\right)^{1/2}
\ge \max_{1\le i\le d} |x_i|\geq  \prod_{i=1}^{d}|x_i|^{\frac 1d}\,.
\]
Thus for any positive number $\alpha>0$, $|x|^{-\alpha
}\leq  \prod_{i=1}^{d}|x_i|^{-\frac{\alpha}{d}}$.

On the other hand,
\begin{eqnarray*}
|x|^\frac{2}{2-\alpha}
&=&\bigg[\sum_{i=1}^d|x_i|^2\bigg]^\frac{1}{2-\alpha}
\ge   \bigg[ \max_{1\le i\le d}  |x_i|^2\bigg]^\frac{1}{2-\alpha}\\
&=&   \max_{1\le i\le d}  |x_i|^\frac{2}{2-\alpha} \ge
\frac{1}{d} \sum_{i=1}^{d}|x_i|^\frac{2}{2-\alpha}.
\end{eqnarray*}
Combining the above with \eqref{e.Y-bound} yields
\eqref{e.Y-product-b}  since the exponents in $|x-\xi|$
in \eqref{e.Y-bound} are negative.
\end{proof}

\begin{Lemma}\label{lem9}
Let $-1<\beta\leq0,  x\in \mathbb{R}$.    Then, there is a constant $C$, dependent on
$\sigma$,  $\alpha$ and $\beta$,  but independent of $\xi$ and $s$ such that
\[
\sup_{\xi\in \mathbb{R }}\int_{\mathbb{R}}|x|^{\beta}
 p(s, x-\xi)d x\leq C s^{\frac{\alpha\beta}{2}+\frac{\alpha}{2}}\,.
\]
\end{Lemma}

\begin{proof}
Making the substitution  $x= {y}{s^{\frac{\alpha}{2}}}$ we obtain
\begin{eqnarray*}
\int_{\mathbb{R}}|x|^{\beta}  p(s, x-\xi)d x
&=&s^{\frac{\alpha\beta}{2}+\frac{\alpha}{2}}\int_{\mathbb{R}}|y|^{\beta}\cdot \exp\bigg(-\sigma\bigg|y-\frac{\xi}{s^{\frac{\alpha}{2}}}\bigg|^{\frac2{2-\alpha}}\bigg)d y\\
&\leq &s^{\frac{\alpha\beta}{2}+\frac{\alpha}{2}} \bigg(\int_{|y|\leq 1}
|y|^{\beta} dy+\int_{\mathbb{R}} \exp\bigg(-\sigma\bigg|y-\frac{\xi}{s^{\frac{\alpha}{2}}}\bigg|^{\frac2{2-\alpha}}
\bigg)d y \bigg)\\
&\leq&  C s^{\frac{\alpha\beta}{2}+\frac{\alpha}{2}}
\end{eqnarray*}
since the two integrals inside the parenthesis  are finite
(and independent of $s$ and $\xi$).
\end{proof}

The following   is a slight extension of the above lemma.
\begin{Lemma}\label{lema11}  There is a constant $C$, dependent on
$\sigma$, $\alpha$ and $\beta$, but independent of $\xi$ and $s$ such that
\[
\sup_{\xi\in \mathbb{R }}\int_{\mathbb{R}}|x|^{\beta}
 \left| \log |x| \right|
 p(s, x-\xi)d x\leq C s^{\frac{\alpha\beta}{2}+\frac{\alpha}{2}} \left[1+\left|\log s\right|\right]\,.
\]
\end{Lemma}

\begin{proof} We shall follow the same idea as in the proof of Lemma \ref{lem9}.
Making the substitution  $x= {y}{s^{\frac{\alpha}{2}}}$ we obtain
\begin{eqnarray*}
&&\int_{\mathbb{R}} |x|^{\beta}
 \left| \log |x| \right|  p(s, x-\xi)d x\\
&  & \qquad \le C s^{\frac{\alpha\beta}{2}+\frac{\alpha}{2}}    \int_{\mathbb{R}}
|y|^{\beta}\left[ |\log |y| |+|\log s | \right]
\cdot \exp\bigg(-\sigma\bigg|y-\frac{\xi}{s^{\frac{\alpha}{2}}}\bigg|^{\frac2{2-\alpha}}\bigg)d y\\
&  & \qquad \le Cs^{\frac{\alpha\beta}{2}+\frac{\alpha}{2}}   (1+|\log s|)\bigg(\int_{|y|\leq e}
|y|^{\beta} |\log |y| | dy+\int_{\mathbb{R}} \exp\bigg(-\sigma\bigg|y-\frac{\xi}{s^{\frac{\alpha}{2}}}\bigg|^{\frac2{2-\alpha}}
\bigg)d y \bigg)\\
&  & \qquad \le     Cs^{\frac{\alpha\beta}{2}+\frac{\alpha}{2}}   (1+|\log s|)\,.
\end{eqnarray*}
This proves the lemma.
\end{proof}



\begin{Lemma}\label{lemma11}
Let   $\theta_1$ and $\theta_2$ satisfy
$-1<\theta_1< 0, -1<\theta_2\leq 0$.  Then for any
$\rho_1, \tau_2 \in \mathbb{R},  \rho_1\neq\tau_2$,
\begin{enumerate}
\item\ If $\theta_1+\theta_2=-1$,   then
\[
\int_{\mathbb{R}}|\rho_1-\tau_1|^{\theta_1}|\rho_2-\rho_1|^{\theta_2}p(s_2-s_1, \rho_2-\rho_1)d\rho_1\leq C+ C|\log(\rho_2-\tau_1)|\,.
\]
\item\  If $\theta_1+\theta_2<-1, $ then
$$\int_{\mathbb{R}}|\rho_1-\tau_1|^{\theta_1}|\rho_2-\rho_1|^{\theta_2}p(s_2-s_1, \rho_2-\rho_1)d\rho_1\leq C|\rho_2-\tau_1|^{1+\theta_1+\theta_2}.$$
\end{enumerate}
\end{Lemma}

\begin{proof}
Without loss of generality,  suppose $\tau_1\leq\rho_2$.  We divide the integral domain into four  intervals.
\begin{eqnarray*}
&&\int_{\mathbb{R}}|\rho_1-\tau_1|^{\theta_1}|\rho_2-\rho_1|^{\theta_2}p(s_2-s_1, \rho_2-\rho_1)d\rho_1\\
&=&\int_{-\infty}^{\frac{3\tau_1-\rho_2}{2}}|\rho_1-\tau_1|^{\theta_1}|\rho_2-\rho_1|^{\theta_2}p(s_2-s_1, \rho_2-\rho_1)d\rho_1\\
&& +\int_{\frac{3\tau_1-\rho_2}{2}}^{\frac{\tau_1+\rho_2}{2}}|\rho_1-\tau_1|^{\theta_1}|\rho_2-\rho_1|^{\theta_2}p(s_2-s_1, \rho_2-\rho_1)d\rho_1\\
&& +\int^{\frac{3\rho_2-\tau_1}{2}}_{\frac{\tau_1+\rho_2}{2}}|\rho_1-\tau_1|^{\theta_1}|\rho_2-\rho_1|^{\theta_2}p(s_2-s_1, \rho_2-\rho_1)d\rho_1\\
&&  +\int_{\frac{3\rho_2-\tau_1}{2}}^{\infty}|\rho_1-\tau_1|^{\theta_1}|\rho_2-\rho_1|^{\theta_2}p(s_2-s_1, \rho_2-\rho_1)d\rho_1\\
&=:&I_1+I_2+I_3+I_4\,.
\end{eqnarray*}
Let us consider $I_2$ first.  When $\displaystyle \rho_1\in
\left[{\frac{3\tau_1-\rho_2}{2}},  {\frac{\tau_1+\rho_2}{2}}\right]$,  we have
$  |\rho_2-\rho_1| \ge \frac{\rho_2-\tau_1}{2}$.   Noticing
$p(s_2-s_1, \rho_2-\rho_1)\le 1$,  we  have the following estimate for $I_2$:
%
\begin{eqnarray*}
I_2&\leq&\bigg(\frac{\rho_2-\tau_1}{2}\bigg)^{\theta_2}\int_{\frac{3\tau_1-\rho_2}{2}}^{\frac{\tau_1+\rho_2}{2}}|\rho_1-\tau_1|^{\theta_1}d\rho_1\\
&\leq& \bigg(\frac{\rho_2-\tau_1}{2}\bigg)^{\theta_2}
\left[
\int_{\tau_1}^{\frac{\tau_1+\rho_2}{2}}(\rho_1-\tau_1)^{\theta_1}d\rho_1
+\int_{\frac{3\tau_1-\rho_2}{2}}
^{\tau_1}(  \tau_1-\rho_1) ^{\theta_1}d\rho_1\right]\\
& = &C\bigg( {\rho_2-\tau_1} \bigg)^{1+\theta_1+\theta_2}\,.
\end{eqnarray*}
With the  same argument, we have
\[
I_3\le
C\bigg( {\rho_2-\tau_1} \bigg)^{1+\theta_1+\theta_2}\,.
\]

Now, we  study $I_1$.  The term  $I_4$ can be analyzed in a similar way.
Since $\rho_1< \frac{3\tau_1-\rho_2}{2}< \tau_1< \rho_2$,  we have
\[
I_1\leq \int_{-\infty}^{\frac{3\tau_1-\rho_2}{2}}(\tau_1-\rho_1)^{\theta_1+\theta_2}p(s_2-s_1, \rho_2-\rho_1)d\rho_1\,.
\]
To estimate the above integral, we divide our estimation into three cases.

\noindent{\sl Case i): \  $\theta_1+\theta_2<-1$}.  \quad

In this case, we bound $p(s_2-s_1, \rho_2-\rho_1)$ by $1$.  Thus, we have
\[
I_1\leq \int_{-\infty}^{\frac{3\tau_1-\rho_2}{2}}(\tau_1-\rho_1)^{\theta_1+\theta_2}d\rho_1= \frac{1}{1+\theta_1+\theta_2}\bigg(\frac{\rho_2-\tau_1}{2}\bigg)^{1+\theta_1+\theta_2}\,.
\]

\noindent{\sl Case ii):\   $\theta_1+\theta_2=-1, \frac{\rho_2-\tau_1}{2}\geq 1$}.

In this case, we have $\frac{3\tau_1-\rho_2}{2}\le \tau_1-1$.  Thus, we have
\begin{eqnarray*}
I_1
&\leq& \int_{-\infty}^{\tau_1-1} (\tau_1-\rho_1)^{-1}p(s_2-s_1, \rho_2-\rho_1)d\rho_1\\
&\leq &\int_{-\infty}^{\tau_1-1} p(s_2-s_1, \rho_2-\rho_1)d\rho_1 \\
&\leq &\int_{-\infty}^{\infty } p(s_2-s_1, \rho_2-\rho_1)d\rho_1
\end{eqnarray*}
which is bounded when $s_1$ and $s_2$ are  in a bounded domain.

\noindent{\sl Case iii): \   $\theta_1+\theta_2=-1, \frac{\rho_2-\tau_1}{2}<1$}.

In this case, we divide the integral into two intervals as follows.
\begin{eqnarray*}
I_1&=& \int_{-\infty}^{\frac{3\tau_1-\rho_2}{2}}(\tau_1-\rho_1)^{\theta_1+\theta_2}p(s_2-s_1, \rho_2-\rho_1)d\rho_1\\
&\leq& \int_{-\infty}^{\tau_1-1} (\tau_1-\rho_1)^{-1}p(s_2-s_1, \rho_2-\rho_1)d\rho_1 + \int_{\tau_1-1}^{\frac{3\tau_1-\rho_2}{2}} (\tau_1-\rho_1)^{-1}p(s_2-s_1, \rho_2-\rho_1)d\rho_1\\
&\leq& C+ \int_{\tau_1-1}^{\frac{3\tau_1-\rho_2}{2}} (\tau_1-\rho_1)^{-1}d\rho_1\\
&\le& C+C|\ln(\rho_2-\tau_1)|\,.
\end{eqnarray*}
%
%
%
Similar argument works for $I_4$.  Combining the estimates for $I_k$,$k=1, 2, 3, 4$ yields the lemma.
\end{proof}

\begin{Lemma}\label{l.13}
Let   $\theta_1$ and $\theta_2$ satisfy
$-1<\theta_1< 0, -1<\theta_2\leq 0$ and $\theta_1+2\theta_2>-2$.
Let $0\le r_1<r_2\le T$ and $0\le s_1<s_2\le T$. Then for any
$\rho_1, \tau_2 \in \mathbb{R},  \rho_1\neq\tau_2$,   we have
\begin{eqnarray}
&&\int_{\mathbb{R}^2}|\rho_1-\tau_1|^{\theta_1}|\rho_2-\rho_1|^{\theta_2}|\tau_2-\tau_1|^{\theta_2}p(s_2-s_1, \rho_2-\rho_1)p(r_2-r_1, \tau_2-\tau_1)d\rho_1 d\tau_1\nonumber\\
&\leq& \left\{
  \begin{array}{ll}
    $$  C (s_2-s_1)^{\frac{\alpha(\theta_1+\theta_2+1)}{2}  }(r_2-r_1)^{\frac{\alpha(\theta_2+1)}{2} }, $$  &  \hbox{$\theta_1+\theta_2>-1$;} \\
   $$ C (r_2-r_1)^{\frac{\alpha(\theta_1+2\theta_2+2)}{2} }, $$  & \hbox{$\theta_1+\theta_2<-1$;} \\
    $$C (r_2-r_1)^{ \frac{\alpha  (\theta_2 +1)}{2}  }
    \left[1+|\log (r_2-r_1)|\right] , $$ & \hbox{ $\theta_1+\theta_2=-1.$}
  \end{array}
\right.\label{e.double.int}
\end{eqnarray}
\end{Lemma}

\begin{proof}
First, we write
\begin{eqnarray}
I&:=&\int_{\mathbb{R}^2}|\rho_1-\tau_1|^{\theta_1}|\rho_2-\rho_1|^{\theta_2}|\tau_2-\tau_1|^{\theta_2}p(s_2-s_1, \rho_2-\rho_1)p(r_2-r_1, \tau_2-\tau_1)d\rho_1d\tau_1\nonumber\\
&=&\int_{\mathbb{R}} f(\tau_1,   \rho_2, s_1, s_2, \theta_1, \theta_2) |\tau_2-\tau_1|^{\theta_2}p(r_2-r_1, \tau_2-\tau_1)d\tau_1\,,\label{e.4.4}
\end{eqnarray}
where
\[
f(\tau_1,   \rho_2, s_1, s_2, \theta_1, \theta_2)= \int_{\mathbb{R}}|\rho_1-\tau_1|^{\theta_1}|\rho_2-\rho_1|^{\theta_2}p(s_2-s_1, \rho_2-\rho_1)d\rho_1 \,.
\]

We divide the situation into three cases.

\noindent{\sl Case i):\    $\theta_1+\theta_2>-1$}.

In this case we apply the  H\"{o}lder's inequality to obtain
\begin{eqnarray}
f(\tau_1,   \rho_2, s_1, s_2, \theta_1, \theta_2)
&\leq& \bigg\{\int_{\mathbb{R}}|\rho_1-\tau_1|^{\theta_1+\theta_2}p(s_2-s_1, \rho_2-\rho_1)d\rho_1 \bigg\}^{\frac{\theta_1}{\theta_1+\theta_2}}\nonumber\\
&&\qquad
\cdot\bigg\{\int_{\mathbb{R}}|\rho_2-\rho_1|^{\theta_1+\theta_2}p(s_2-s_1, \rho_2-\rho_1)d\rho_1\bigg\}^{\frac{\theta_2}{\theta_1+\theta_2}} \nonumber\\
&\leq&  C (s_2-s_1)^{\frac{\alpha(\theta_1+\theta_2)}{2}+\frac{\alpha}{2}} \,,
\label{e.4.5}
\end{eqnarray}
where the last inequality follows from   Lemma \ref{lem9}.
Substituting the above estimate \eqref{e.4.5} into \eqref{e.4.4},  we have
\begin{eqnarray*}
I&=& \int_{\mathbb{R}} f(\tau_1,   \rho_2, s_1, s_2, \theta_1, \theta_2) |\tau_2-\tau_1|^{\theta_2}p(r_2-r_1, \tau_2-\tau_1)d\tau_1\\
&\le&  C (s_2-s_1)^{\frac{\alpha(\theta_1+\theta_2)}{2}+\frac{\alpha}{2}}    \int_{\mathbb{R}} |\tau_2-\tau_1|^{\theta_2}p(r_2-r_1, \tau_2-\tau_1)d\tau_1\,.
\end{eqnarray*}

Using Lemma \ref{lem9} again     we have,  
\[
I\leq C (s_2-s_1)^{\frac{\alpha(\theta_1+\theta_2)}{2}+ \frac{\alpha}{2} }(r_2-r_1)^{\frac{\alpha\theta_2}{2}+ \frac{\alpha}{2} }\,.
\]

\noindent{\sl Case ii):\
 $\theta_1+\theta_2<-1$}.

In this case,  from Lemma \ref{lemma11},  part (ii)  \  it follows
\[
f(\tau_1,   \rho_2, s_1, s_2, \theta_1, \theta_2) \le C |\rho_2-\tau_1|^{\theta_1+\theta_2+1} \,.
\]
Hence, we have
\begin{eqnarray*}
I
&\le&
  C \int_{\mathbb{R}} |\rho_2-\tau_1|^{\theta_1+\theta_2+1} |\tau_2-\tau_1|^{\theta_2}p(r_2-r_1, \tau_2-\tau_1)d\tau_1\,.
\end{eqnarray*}
Now,  since from the condition of the lemma,
$\theta_1+2\theta_2+1>-1$,   we can use H\"{o}lder's inequality
such as in the inequality \eqref{e.4.5}   in the case (i),  to obtain
\begin{eqnarray*}
I
&\le& C \int_{\mathbb{R}} |\rho_2-\tau_1|^{\theta_1+\theta_2+1} |\tau_2-\tau_1|^{\theta_2}p(r_2-r_1, \tau_2-\tau_1)d\tau_1\\
&\leq &  C (r_2-r_1)^{\frac{\alpha(\theta_1+2\theta_2)}{2}+ \alpha }\,.
\end{eqnarray*}

\noindent{\sl Case iii):\      $\theta_1+\theta_2=-1$.}

In this case, we first use  Lemma \ref{lemma11},  part (i)  to obtain
\[
f(\tau_1,   \rho_2, s_1, s_2, \theta_1, \theta_2) \le
C  \left[1+  \left|\log|\rho_2-\tau_1|\right|\right]  \,.
\]
Thus, using Lemma \ref{lema11}, we have
\begin{eqnarray*}
I
&\leq& C \int_{\mathbb{R}}\{1+  \left|\log|\rho_2-\tau_1|
\right| \} |\tau_2-\tau_1|^{\theta_2}p(r_2-r_1, \tau_2-\tau_1)d\tau_1\\
&\leq & C (r_2-r_1)^{\frac{\alpha(\theta_2+1)}{2}  }\left[
1+\left|\log |r_2-r_1|\right| \right]\,.
\end{eqnarray*}
The lemma is then proved.
\end{proof}

\begin{Corollary}\label{c.14}
Let   $\theta_1$ and $\theta_2$ satisfy
$-1<\theta_1< 0, -1<\theta_2\leq 0$ and $\theta_1+2\theta_2>-2$.
Let $0\le r_1<r_2\le T$ and $0\le s_1<s_2\le T$.   Then for any
$\rho_1, \tau_2 \in \mathbb{R},  \rho_1\neq\tau_2$,   we have
\begin{eqnarray}
&&\int_{\mathbb{R}^2}|\rho_1-\tau_1|^{\theta_1}|\rho_2-\rho_1|^{\theta_2}|\tau_2-\tau_1|^{\theta_2}p(s_2-s_1, \rho_2-\rho_1)p(r_2-r_1, \tau_2-\tau_1)d\rho_1 d\tau_1\nonumber\\
&\leq&
\begin{cases}
C (s_2-s_1)^{\frac{\alpha(\theta_1+2\theta_2+2)}{4}  }(r_2-r_1)^{\frac{\alpha(\theta_1+2\theta_2+2)}{4}}\,;& \qquad \theta_1+\theta_2\not =-1\\
C (s_2-s_1)^{\frac{\alpha( \theta_2+1)}{4}  }(r_2-r_1)^{\frac{\alpha( \theta_2+1)}{4}}
\left[1+|\log(r_2-r_1)|+|\log(s_2-s_1)|\right]\,;& \qquad \theta_1+\theta_2  =-1\,.
\end{cases} \nonumber\\
\label{e.the-integral}
\end{eqnarray}
\end{Corollary}
\begin{proof} Consider first the case $\theta_1+\theta_2<-1$. Denote the integral
on the left hand side of \eqref{e.the-integral}  by $I$.  Then
the inequality   \eqref{e.the-integral}
implies
\[
I\le C (r_2-r_1)^{\frac{\alpha(\theta_1+2\theta_2)}{2} + \alpha }\,.
\]
In the same way we have
\[
I\le C (s_2-s_1)^{\frac{\alpha(\theta_1+2\theta_2)}{2} + \alpha }\,.
\]
Now we use the fact that if three numbers satisfying
$a\le b$ and $a\le c$, then $a= a^{1/2}a^{1/2} \le b^{1/2}c^{1/2}$.
\[
I\le C (r_2-r_1)^{\frac{\alpha(\theta_1+2\theta_2)}{4} + \alpha/2 }
 (s_2-s_1)^{\frac{\alpha(\theta_1+2\theta_2)}{4} + \alpha/2 }
\]
which simplifies to \eqref{e.the-integral}.   The exactly the same argument
applied to the case $\theta_1+\theta_2=-1$ and the case
$\theta_1+\theta_2>-1$. Thus, the inequality \eqref{e.4.5} implies  \eqref{e.the-integral}.
\end{proof}

\begin{Lemma}\label{fromgr} Let $p_1,\cdots, p_n>0$.  Then for any $T>0$,
\begin{equation}
\int_{0\le s_1<\cdots<s_n\le T}  (s_n-s_{n-1})^{p_{n}-1}\cdots (s_2-s_1)^{p_2-1} s_1^{p_1-1}
ds=\frac{T^n\prod_{k=1}^n \Gamma(p_k)}{\Gamma(p_1+\cdots+p_n+1)}\,.
\end{equation}
\end{Lemma}
\begin{proof}This is well-known. For example it is straightforward consequence of
formula 4.634 of \cite{GR} with  some  obvious transformations.
\end{proof}

\begin{Lemma}\label{lem6} Assume that $u_0$ is bounded.  Then
\[
\sup_{x\in \mathbb{R }}\int_{\mathbb{R}^d}Z(t, x, \xi)u_0(\xi)d\xi \leq C\,.
\]
\end{Lemma}

\begin{proof}  We use
$Z(t, x, \xi)=Z_0(t, x-\xi)+V_Z(t, x, \xi)$.
Since $u_0$ is bounded,
\begin{eqnarray*}
\left|\int_{\mathbb{R}^d}Z_0(t, x, \xi)u_0(\xi)d\xi \right|
&\le& C\int_{\mathbb{R}^d}|Z_0(t, x, \xi) |d\xi
\end{eqnarray*}
which is  bounded  by the estimates in  \eqref{e.ine.z0}
and a substitution $\xi=x+t^{\frac{\alpha}{2}} y$.
In fact, we
have,  for example, when $d\ge 3$,
\begin{eqnarray*}
\int_{\mathbb{R}^d}|Z_0(t, x- \xi ) |d\xi
&\le& C \int_{\mathbb{R}^d}  t^{-\alpha}t^{\frac{(2-d)\alpha}{2}}t^{\frac{d\alpha}{2}}|y|^{2-d}\exp\{-\sigma|y|^{\frac{2}{2-\alpha}}\} dy \leq Ct^{1-\alpha}\leq C.
\end{eqnarray*}

Similarly,  using the estimation for  $V_Z(t, x, \xi)$ given in Proposition \ref{Proposition5}
we can bound $\int_{\RR^d} |V_Z(t, x, \xi)|d\xi$ by a constant. In fact, for example, when $d=3$, we have
\begin{eqnarray*}
\int_{\mathbb{R}^d}|V_Z(t, x, \xi) |d\xi
&\le& C t^{\frac{\gamma_0\alpha}{2}-\alpha} \int_{\mathbb{R}^d}  t^{\frac{3\alpha}{2}}t^{\frac{(\gamma-\gamma_0-1)\alpha}{2}}|y|^{\gamma-\gamma_0-1}\exp\{-\sigma|y|^{\frac{2}{2-\alpha}}\} dy \leq Ct^{\frac{\gamma\alpha}{2}}\leq C.
\end{eqnarray*}
The other dimension case can be dealt with the same  way.
\end{proof}



%
%
%

\section{ Proof of the main theorem \ref{main}.}

Change $t$ to $s$ and $x$ to $y$ and   the equation \eqref{e.2.5} for mild solution becomes
\[
u(s, y)=\int_{\mathbb{R}^d}Z(s, y,\xi)u_0(\xi)d\xi+\int_{0}^s\int_{\mathbb{R}^d}Y(s-r, y, z)u(r, z)W^H(dz )d r\,.
\]
Substituting  the above into \eqref{e.2.5}, we have
\begin{eqnarray*}
u(t,x)&=&
\int_{\mathbb{R}^d}Z(t,  x,\xi)u_0(\xi)d\xi+
\int_{0}^t\int_{\mathbb{R}^{2d} }Y(t-s, x, y)Z(s, y,\xi)u_0(\xi)d\xi W^H(dy )d s\\
&&\qquad
+\int_{0}^t\int_{0}^s\int_{\mathbb{R}^{2d} }Y(t-s, x, y) Y(s-r, y, z)u(r, z)W^H(dz )d rW^H(dy )d s\,.
\end{eqnarray*}
We continue to  iterate  this procedure to obtain
\begin{equation}
u(t, x)= \sum_{n=0}^{\infty} \Psi_n (t, x)\,,\label{e.5.6}
\end{equation}
where $\Psi_n$ satisfies the following recursive relation:
\begin{eqnarray*}
\Psi_0(t, x)&=&\int_{\mathbb{R}^d}Z(t, x, \xi)u_0(\xi)d\xi\\
\Psi_{n+1} (t, x)&=&\int_0^t\int_{\mathbb{R}^d}Y(t-s, x, y)\Psi_n (s, y) W^H(dy)d s\,,\quad n=0, 1, 2, \cdots\
\end{eqnarray*}
To write down the explicit expression for the expansion \eqref{e.5.6}, we denote
\begin{equation}
f_n(t, x;x_1,\cdots , x_n)=\int_{T_n}\int_{\mathbb{R}^{d}}Y(t-s_n, x, x_n)\cdots Y(s_2-s_1, x_2, x_1)Z(s_1, x_1, \xi)u_0(\xi)d\xi d \textbf{s}\,,\label{e.5.a7}
\end{equation}
where
\[
T_n={0< s_1<s_2< \cdots< s_n\leq t}  \quad
{\rm and }\quad
d\textbf{s}=ds_1ds_2\cdots ds_n\,.
\]
With these notations, we  see from the above  iteration procedure that
\begin{eqnarray}
\Psi_{n} (t, x)&=&I_n(\tilde{f}_n (t, x))\nonumber\\
&=&\int_{\mathbb{R}^{nd}}f_n(t, x;x_1,\cdots , x_n)  W^H(dx_1)  W^H(dx_2)\cdots   W^H(dx_n)\nonumber\\
&=& \int_{\mathbb{R}^{nd}}\tilde f_n(t, x;x_1,\cdots , x_n)  W^H(dx_1)  W^H(dx_2)\cdots   W^H(dx_n)\,. \label{e.5.7}
\end{eqnarray}
where $I_n$ denotes the multiple It\^o type integral with respect to
 $W(x)$ (see \cite{ref-journal2})   and
  $\tilde{f}_n(t, x;x_1,\cdots , x_n)$ is the symmetrization of $f_n(t, x;x_1,\cdots , x_n)$ with respect to $x_1,\cdots , x_n$:
\[
\tilde{f}_n(t, x;x_1,\cdots , x_n)=\frac{1}{n!}
\sum_{  i_1, \cdots, i_n\in \sigma(n)} f_n(t, x; x_{i_1},\cdots , x_{i_n}) \,,
\]
where $\sigma(n)$ denotes the set of permutations of $(1, 2, \cdots, n)$.

The expansion \eqref{e.5.6}  with  the explicit expression \eqref{e.5.7} for $\Psi_n$
is called the {\it chaos expansion of the solution}.

If the equation \eqref{1.4} has a square integrable solution,   then
it has a chaos expansion according to a  general theorem of It\^o.
From the above iteration procedure, it is easy to see that this chaos expansion
of the solution is given uniquely by  \eqref{e.5.6}-\eqref{e.5.7}.
This is the uniqueness.

If we can show that the series \eqref{e.5.6} is convergent in $L^2(\Omega, \mathcal{F}, P)$,
then it is easy to  verify that $u(t,x)$ defined by \eqref{e.5.6}-\eqref{e.5.7}
satisfies the equation \eqref{e.2.5}.   Thus, the existence of the solution to
\eqref{1.4} is solved and the explicit form of the solution is also given
(by \eqref{e.5.6}-\eqref{e.5.7}).  We refer to \cite{ref-journal2} for more detail.

Thus, our remaining   task   is to prove that  the series defined by \eqref{e.5.6} is convergent
in $L^2(\Omega, \mathcal{F}, P)$.  To this end, we   need to use the lemmas that we just proved.

Let now $u(t,x)$  be defined by \eqref{e.5.6}-\eqref{e.5.7}.
Then we have
\begin{eqnarray}
\EE[u(t, x)^2]
&=&\sum_{n=0}^\infty  \EE \left[I_n(\tilde f_n(t, x))\right]^2 \nonumber\\
&=&\sum_{n=0}^\infty n!\langle \tilde{f}_n, \tilde{f}_n\rangle_H\nonumber\\
&\le&
 \sum_{n=0}^\infty n!\langle f_n, f_n\rangle_H\,, \label{e.5.8}
\end{eqnarray}
where
\begin{equation}
\langle f, g\rangle_H=\int_{\mathbb{R}^{2nd}}\prod_{i=1}^n \varphi_H(u_i, v_i)f(u_1, \cdots, u_n)g(v_1, \cdots, v_n)d u_1dv_1du_2dv_2\cdots du_n dv_n
\label{e.5.9}
\end{equation}
and the last inequality follows from H\"older inequality.
Here and in the remaining part of the paper, we use the following notations:
\begin{eqnarray*}
u_i&=& (u_{i1}, \cdots, u_{id})\,,\quad du_i=du_{i1}\cdots du_{id}\,,\quad i=1, 2, \cdots, n\,; \\
\varphi_H(u_i, v_i)&=&\prod_{j=1}^{d}\varphi_{H_j}(u_{ij}, v_{ij})=\prod_{j=1}^{d}H_j(2H_j-1)|u_{ij}-v_{ij}|^{2H_{j}-2}\,.
\end{eqnarray*}

We   use  the   idea in \cite{ref-journal2} to estimate
each term $\Theta_n(t, x)=n!\langle f, f\rangle_H$ in the series \eqref{e.5.8}.
%
By the defining  formula \eqref{e.5.a7}   for  $f_n$ we have
\begin{eqnarray*}
\Theta_{n}(t, x)&=&n!\int_{T^2_n}\int_{\mathbb{R}^{2nd+2}}\prod_{k=1}^n
\varphi_H(\xi_i-\eta_i)\  Y(t-s_n, x, \xi_n)\cdots Y(s_2-s_1, \xi_2, \xi_1)
\\
&&\qquad
\cdot\int_{\mathbb{R}^d}Z(s_1, \xi_1,   \xi_0)u_0(\xi_0) d\xi_0  \cdot Y(t-r_n, x, \eta_n)\cdots Y(r_2-r_1, \eta_2, \eta_1)\\
&&\qquad
\cdot \int_{\mathbb{R}^d}Z(r_1, \eta_1,  \eta_0)u_0(\eta_0) d\eta_0  d\xi d\eta d\textbf{s}  d\textbf{r}.
 \end{eqnarray*}
Application of Lemma \ref{lem6} to the above integral yields
\begin{eqnarray*}
\Theta_{n}(t, x)
&\leq& C n!\int_{T^2_n}\int_{\mathbb{R}^{2nd}}\prod_{k=1}^n\varphi_H(\xi_i-\eta_i)Y(t-s_n, x, \xi_n)\cdots Y(s_2-s_1, \xi_2, \xi_1)\\
&&\qquad \cdot Y(t-r_n, x, \eta_n)\cdots Y(r_2-r_1, \eta_2, \eta_1) d\xi d\eta d\textbf{s}  d\textbf{r}.
 \end{eqnarray*}
Using  Lemma \ref{Corollary7} to the above integral,  we have
\begin{equation}\label{4.2}
\Theta_{n}(t, x)\leq C^n  n!\int_{T_n^2}\prod_{i=1}^d \Theta_{i, n}(t, x_i,\textbf{s},\textbf{r}) d\textbf{s} d\textbf{r},
\end{equation}
where
\begin{eqnarray*}
\Theta_{i, n}(t, x_i,\textbf{s},\textbf{r})
&=&\int_{\mathbb{R}^{2n}}\left\{\prod_{k=1}^n\varphi_{H_i}(\rho_k-\tau_k)\right\}
|t-s_n|^{\frac{\zeta_{d}}{d}}|x_i-\rho_n|^{\frac{\kappa_{d}}{d}}  p(t-s_n, x_i-\rho_n)  \\
&&\qquad  \cdots |s_2-s_1|^{\frac{\zeta_{d}}{d}}|\rho_2-\rho_1|^{\frac{\kappa_{d}}{d}}  p(s_2-s_1, \rho_2-\rho_1)\\
&&\qquad \cdot |t-r_n|^{\frac{\zeta_{d}}{d}}  |x_i-\tau_n|^{\frac{\kappa_{d}}{d}}  p(t-r_n, x_i-\tau_n)  \cdots|r_2-r_1|^{\frac{\zeta_{d}}{d}} \\
&&\qquad \cdot |\tau_2-\tau_1|^{\frac{\kappa_{d}}{d}}  p(r_2-r_1, \tau_2-\tau_1) d\rho d\tau .
\end{eqnarray*}
Here we use the notation $\rho_k=\xi_{ki}$ and $\tau_k=\eta_{ki}$, $k=1, \cdots, n$.
The quantity $\Theta_{i,n}$ can be written as
\begin{eqnarray}
\Theta_{i, n}(t, x_i,\textbf{s},\textbf{r})
&=&  |t-s_n|^{\frac{\zeta_{d}}{d}}  |t-r_n|^{\frac{\zeta_{d}}{d}}
  \cdots |s_2-s_1|^{\frac{\zeta_{d}}{d}}     |r_2-r_1|^{\frac{\zeta_{d}}{d}}
  \nonumber\\
&&\qquad \cdot \int_{\mathbb{R}^{2n}}\left\{\prod_{k=1}^n\varphi_{H_i}(\rho_k-\tau_k)\right\}
 |x_i-\rho_n|^{\frac{\kappa_{d}}{d}}  p(t-s_n, x_i-\rho_n)   \nonumber\\
&&\qquad  \cdot |x_i-\tau_n|^{\frac{\kappa_{d}}{d}}  p(t-r_n, x_i-\tau_n)  \cdots  |\rho_2-\rho_1|^{\frac{\kappa_{d}}{d}}  p(s_2-s_1, \rho_2-\rho_1)    \nonumber\\
&&\qquad
\cdots  |\tau_2-\tau_1|^{\frac{\kappa_{d}}{d}}  p(r_2-r_1, \tau_2-\tau_1) d\rho d\tau \,.\label{e.5.12}
\end{eqnarray}
From the definition \eqref{e.kappa}   of $\kappa_d$ we see easily
$
\frac{\kappa_{d}}{d}>-1\,. $\
We assume
\begin{equation}
2H_i +\frac{2\kappa_{d}}{d}>0\,.\label{e.Hi}
\end{equation}
Under the above condition  we can apply  the Corollary   \ref{c.14}   with  $\theta_1=2H_i-2>-1$, $\theta_2  =\frac{\kappa_{d}}{d}>-1$
 to the integration  $d\rho_1d\tau_1$ in  the expression  \eqref{e.5.12}
 (Condition \eqref{e.Hi} implies that $\theta_1+2\theta_2>-2$).  Then,
 when $\theta_1+\theta_2\not=-1$,  we have
\begin{eqnarray}
\Theta_{i, n}(t, x_i,\textbf{s},\textbf{r})
&\le &  C|t-s_n|^{\frac{\zeta_{d}}{d}}  |t-r_n|^{\frac{\zeta_{d}}{d}}
  \cdots |s_3-s_2|^{\frac{\zeta_{d}}{d}}     |r_3-r_2|^{\frac{\zeta_{d}}{d}}   \nonumber\\
&&\qquad \cdot |s_2-s_1|^{\frac{\zeta_{d}}{d}+\frac{H_id+\kappa_d}{2d}\alpha}     |r_2-r_1|^{\frac{\zeta_{d}}{d}+ \frac{H_id+\kappa_d}{2d}\alpha }
  \nonumber\\
&&\qquad \cdot \int_{\mathbb{R}^{2n-2}}\left\{\prod_{k=2}^n\varphi_{H_i}(\rho_k-\tau_k)\right\}
 |x_i-\rho_n|^{\frac{\kappa_{d}}{d}}  p(t-s_n, x_i-\rho_n)   \nonumber\\
&&\qquad  \cdot |x_i-\tau_n|^{\frac{\kappa_{d}}{d}}  p(t-r_n, x_i-\tau_n)  \cdots  |\rho_3-\rho_2|^{\frac{\kappa_{d}}{d}}  p(s_3-s_2, \rho_3-\rho_2)    \nonumber\\
&&\qquad
\cdots  |\tau_3-\tau_3|^{\frac{\kappa_{d}}{d}}  p(r_3-r_2, \tau_3-\tau_2) d\rho_n\cdots d\rho_2 d\tau_n\cdots d\tau_2 \,.\ \nonumber
\end{eqnarray}
Repeatedly applying this argument, we obtain
\begin{eqnarray}
\Theta_{i, n}(t, x_i,\textbf{s},\textbf{r})
&\le &  C^n \prod_{k=1}^n |t_{k+1}-t_k|^{\ell_i }    |s_{k+1}-s_k|^{\ell_i }    \,,
\end{eqnarray}
where we recall the convention that  $t_{n+1}=t$ and  $s_{n+1}=s$  and where
\begin{equation*}
\ell_i=\frac{\zeta_{d}}{d}+\frac{H_i d+\kappa_d}{2d}\alpha \,.
\end{equation*}
%
Substituting the above estimate of $\Theta_{i,n}$ into the expression for
$\Theta_n$,  we have
\begin{eqnarray*}
\Theta_n(t, x)
&\leq&  C^n \int_{T_n^2}\prod_{k=1}^{n} (s_{k+1}-s_k)^{ \ell  }(r_{k+1}-r_k)^{ \ell  } d\textbf{s} d\textbf{r}\\
&=& C^n \left[ \int_{T_n }\prod_{k=1}^{n} (s_{k+1}-s_k)^{ \ell  }  d\textbf{s}\right]^2\,,
\end{eqnarray*}
where
\begin{equation*}
\ell =\sum_{i=1}^d \ell_i= \zeta_{d} + \frac{|H|\alpha}{2}     +\frac{\kappa_d\alpha }{2 }   \quad {\rm with}\quad
|H|=\sum_{i=1}^d H_i\,.
\end{equation*}

Now, we apply Lemma \ref{fromgr} to obtain
\begin{eqnarray*}
\Theta_n(t, x)
&\leq&    C^n  \left[\frac{\Gamma(\ell+1) }{\Gamma(n(\ell+1))}  \right]^2\nonumber \\
&\leq&     \frac{C^n  }{\Gamma(2n(\ell+1))}   \,.
\end{eqnarray*}
This estimate combined with \eqref{e.5.8} proves that if
\begin{equation}
2(\ell+1)>1\,, \label{e.5.cond2}
\end{equation}
then  $\sum_{n=0}^\infty \Theta_n(t, x)$ is bounded which implies that the series
\eqref{e.5.6} is convergent in $L^2(\Omega, \mathcal{F}, P)$.

From  the explicit expressions of $\zeta_d$ and $\kappa_d$  we see by analyzing the condition
\eqref{e.5.cond2}  for the cases
$d=1$, $d=2$, $d=4$ and $d=3$ or $d\ge 5$ separately.  We see that
the condition \eqref{e.5.cond2}  is equivalent to
\begin{equation}
\sum_{i=1}^d H_i>d-2+\frac{1}{\alpha}\,.\label{e.5.17}
\end{equation}

When $\theta_1+\theta_2=-1$,  Corollary \ref{c.14} implies that for any $\varepsilon>0$,\begin{eqnarray*}
&&\int_{\mathbb{R}^2}|\rho_1-\tau_1|^{\theta_1}|\rho_2-\rho_1|^{\theta_2}|\tau_2-\tau_1|^{\theta_2}p(s_2-s_1, \rho_2-\rho_1)p(r_2-r_1, \tau_2-\tau_1)d\rho_1 d\tau_1\nonumber\\
&\leq&
C (s_2-s_1)^{\frac{\alpha( \theta_2+1+\varepsilon )}{4}  }(r_2-r_1)^{\frac{\alpha( \theta_2+1+
+\varepsilon)}{4}} \,.
 \end{eqnarray*}
Now we can follow  the above same argument to obtain that if
\begin{equation}
 2(\ell+1)>1\,,\label{e.5.18}
\end{equation}
where $\ell=\frac{d\epsilon+\kappa_d+d }{4}\alpha$,
then
$\Theta_n(t, x)$ is bounded.  In the same way as in the case $\theta_1+\theta_2\not=-1$,
we can show that the condition \eqref{e.5.17} implies \eqref{e.5.18}.

Now we consider the condition \eqref{e.Hi}.  From the definition
\eqref{e.kappa} of $\kappa_d$, we see that when $d=1,2, 3, 4$,
$H_i>1/2$ implies \eqref{e.Hi}.  When $d\ge 5$, then
the condition \eqref{e.Hi} is implied by the following
\[
H_i>1-\frac{2}{d}-\frac{\gamma}{2d}
\]
by choosing $\gamma_0$ sufficiently small.
%
%
%
Theorem 2 is then proved.\quad $\Box$.

\bibliographystyle{mdpi}
\makeatletter
\renewcommand\@biblabel[1]{#1. }
\makeatother



%


%

\end{document}